\documentclass[12pt]{amsart}
\usepackage[latin1]{inputenc}
\usepackage{amsbsy}
\usepackage{amscd} 
\usepackage{amsfonts}
\usepackage{amsgen} 
\usepackage{amsmath}
\usepackage{amsopn} 
\usepackage{amssymb}
\usepackage{amstext}
\usepackage{amsthm} 
\usepackage{amsxtra}
\usepackage[all]{xy}
\usepackage{booktabs}
\usepackage{rotating}
\usepackage{tikz}
\usepackage{scalerel}
\usepackage{float}

\theoremstyle{plain} 
\newtheorem{thm}{Theorem}[section]

\newtheorem{prop}[thm]{Proposition}
\newtheorem{lem}[thm]{Lemma}
\newtheorem{cor}[thm]{Corollary}
\theoremstyle{definition}
\newtheorem{defn}[thm]{Definition}
\newtheorem{rem}[thm]{Remark}

\numberwithin{equation}{section}

\renewcommand{\theta}{\vartheta}
\renewcommand{\phi}{\varphi}
\renewcommand{\epsilon}{\varepsilon}
\renewcommand{\subset}{\subseteq}

\newcommand{\Z}{\mathbb Z}

\newcommand{\C}{\mathbb C}

\newcommand{\Aut}{G_{aut}}

\newcommand{\QBan}{G_{aut}^+}

\begin{document}

\title{Quantum automorphisms of folded cube graphs}
\author{Simon Schmidt}
\address{Saarland University, Fachbereich Mathematik, 
66041 Saarbr\"ucken, Germany}
\email{simon.schmidt@math.uni-sb.de}
\thanks{The author is supported by the DFG project \emph{Quantenautomorphismen von Graphen}. The author is grateful to his supervisor Moritz Weber for proofreading the article and for numerous discussions on the topic. He also wants to thank Julien Bichon and David Roberson for helpful comments and suggestions. Furthermore, he thanks the referee for useful comments and for correcting a mistake in the proof of Theorem \ref{noncomm}.}
\date{\today}
\subjclass[2010]{46LXX (Primary); 20B25, 05CXX (Secondary)}
\keywords{finite graphs, graph automorphisms, automorphism groups, quantum automorphisms, quantum groups, quantum symmetries}

\begin{abstract}
We show that the quantum automorphism group of the Clebsch graph is $SO_5^{-1}$. This answers a question by Banica, Bichon and Collins from 2007. More general for odd $n$, the quantum automorphism group of the folded $n$-cube graph is $SO_n^{-1}$. Furthermore, we show that if the automorphism group of a graph contains a pair of disjoint automorphisms this graph has quantum symmetry. 
\end{abstract}

\maketitle

\section*{Introduction}

The concept of quantum automorphism groups of finite graphs was introduced by Banica and Bichon in \cite{QBan,QBic}. It generalizes the classical automorphism groups of graphs within the framework of compact matrix quantum groups. We say that a graph has no quantum symmetry if the quantum automorphism group coincides with its usual automorphism group. A natural question is: When does a graph have no quantum symmetry? This has been studied in \cite{Che} for some graphs on $p$ vertices, $p$ prime, and more recently the author showed in \cite{QAutPetersen} that the Petersen graph does not have quantum symmetry. Also Lupini, Man\v{c}inska and Roberson proved that almost all graphs have trivial quantum automorphim group in \cite{Nonlocal}, which implies that almost all graphs do not have quantum symmetry.

In this article we develop a tool for detecting quantum symmetries namely we show that a graph has quantum symmetry if its automorphism group contains a pair of disjoint automorphisms (Theorem \ref{noncomm}). As an example, we apply it to the Clebsch graph and obtain that it does have quantum symmetry (Corollary \ref{C}).

We even go further and prove that the quantum automorphism group of the Clebsch graph is $SO_5^{-1}$, the $q$-deformation at $q=-1$ of $SO_5$, answering a question from \cite{survey}. For this we use the fact that the Clebsch graph is the folded 5-cube graph. This can be pushed further to more general folded $n$-cube graphs: In \cite{hyperoctahedral}, two generalizations of the hyperoctahedral group $H_n$ are given, one of them being $O_n^{-1}$ as quantum symmetry group of the hypercube graph. To prove that $O_n^{-1}$ is the quantum symmetry group of the hypercube graph, Banica, Bichon and Collins used the fact that the hypercube graph is a Cayley graph. It is also well known that the folded cube graph is a Cayley graph. We use similar techniques as in \cite{hyperoctahedral} to show that for odd $n$,  the quantum symmetry group of the folded $n$-cube graph is $SO_n^{-1}$ which is the quotient of $O_n^{-1}$ by some quantum determinant condition (Theorem \ref{main}). This constitutes our main result. 

\section{Preliminaries}

\subsection{Compact matrix quantum groups}

We start with the definition of compact matrix quantum groups which were defined by Woronowicz \cite{CMQG1,CMQG2} in 1987. See \cite{Nesh, Tim} for recent books on compact quantum groups. 

\begin{defn}
A \emph{compact matrix quantum group} $G$ is a pair $(C(G),u)$, where $C(G)$ is a unital (not necessarily commutative) $C^*$-algebra which is generated by $u_{ij}$, $1 \leq i,j \leq n$, the entries of a matrix $u \in M_n(C(G))$. Moreover, the *-homomorphism $\Delta: C(G) \to C(G) \otimes C(G)$, $u_{ij} \mapsto \sum_{k=1}^n u_{ik} \otimes u_{kj}$ must exist, and $u$ and its transpose $u^{t}$ must be invertible. 
\end{defn}

An important example of a compact matrix quantum group is the quantum symmetric group $S_n^+$ due to Wang \cite{WanSn}. It is the compact matrix quantum group, where
\begin{align*}
C(S_n^+) := C^*(u_{ij}, \, 1 \leq i,j \leq n \, | \, u_{ij} = u_{ij}^* = u_{ij}^2, \, \sum_{l} u_{il} = \sum_{l} u_{li} =1).
\end{align*} 
 
An action of a compact matrix quantum group on a $C^*$-algebra is defined as follows (\cite{Pod, WanSn}).

\begin{defn} 
Let $G=(C(G), u)$ be a compact matrix quantum group and let $B$ be a $C^*$-algebra. 
A \emph{(left) action} of $G$ on $B$ is a unital *-homomorphism $\alpha: B \to B\otimes C(G)$ such that
\begin{itemize}
\item[(i)] $(\mathrm{id} \otimes \Delta ) \circ \alpha = (\alpha \otimes \mathrm{id}) \circ \alpha$
 \item[(ii)] $\alpha(B)(1 \otimes C(G))$ is linearly dense in $B \otimes C(G)$. 
\end{itemize} 
\end{defn}

 In \cite{WanSn}, Wang showed that $S_n^+$ is the universal compact matrix quantum group acting on $X_n = \{1,\dots,n\}$. This action is of the form $\alpha: C(X_n) \to C(X_n) \otimes C(S_n^+)$, 
\begin{align*}
\alpha(e_i) = \sum_{j} e_j \otimes u_{ji}. 
\end{align*}

\subsection{Quantum automorphism groups of finite graphs}
In 2005, Banica \cite{QBan} gave the following definition of a quantum automorphism group of a finite graph.

\begin{defn}
Let $\Gamma =(V, E)$ be a finite graph with $n$ vertices and adjacency matrix $\varepsilon \in M_n(\{0,1\})$.  The \emph{quantum automorphism group} $\QBan(\Gamma)$ is the compact matrix quantum group $(C(\QBan(\Gamma)),u)$, where $C(\QBan(\Gamma))$ is the universal $C^*$-algebra with generators $u_{ij}, 1 \leq i,j \leq n$ and relations 
\begin{align}
&u_{ij} = u_{ij}^* = u_{ij}^2 &&1 \leq i,j \leq n,\label{QA1}\\ 
&\sum_{l=1}^n u_{il} = 1 = \sum_{l=1}^n u_{li}, &&1 \leq i \leq n,\label{QA2}\\
&u \varepsilon = \varepsilon u \label{QA3},
\end{align}
where \eqref{QA3} is nothing but $\sum_ku_{ik}\epsilon_{kj}=\sum_k\epsilon_{ik}u_{kj}$.
\end{defn}

\begin{rem}
There is another definition of a quantum automorphism group of a finite graph by Bichon in \cite{QBic}, which is a quantum subgroup of $\QBan(\Gamma)$. But this article concerns $\QBan(\Gamma)$. See \cite{SWe} for more on quantum automorphism groups of graphs. 
\end{rem}

The next definition is due to Banica and Bichon \cite{BanBic}. We denote by $\Aut(\Gamma)$ the usual automorphism group of a graph $\Gamma$. 

\begin{defn}
Let $\Gamma = (V,E)$ be a finite graph. We say that $\Gamma$ has \emph{no quantum symmetry} if $C(\QBan(\Gamma))$ is commutative, or equivalently
\begin{align*}
C(\QBan(\Gamma)) = C(\Aut(\Gamma)).
\end{align*}
If $C(\QBan(\Gamma))$ is non-commutative, we say that $\Gamma$ has \emph{quantum symmetry}. 
\end{defn} 
Note that $\Aut(\Gamma) \subset \QBan(\Gamma)$, so in general a graph $\Gamma$ has more quantum symmetries than symmetries. 

\subsection{Compact matrix quantum groups acting on graphs}

An action of a compact matrix quantum group on a graph is an action on the functions on the vertices, but with additional structure. This concept was introduced by Banica and Bichon \cite{QBan,QBic}. 

\begin{defn}
Let $\Gamma = (V,E)$ be a finite graph and $G$ be a compact matrix quantum group. An \emph{action of $G$ on $\Gamma$} is an action of $G$ on $C(V)$ such that the magic unitary matrix $(v_{ij})_{1 \leq i,j \leq |V|}$ associated to the formular 
\begin{align*}
\alpha(e_i) = \sum_{j=1}^{|V|} e_j \otimes v_{ji}
\end{align*}
commutes with the adjacency matrix, i.e $v\varepsilon = \varepsilon v$. 
\end{defn}

\begin{rem}
If $G$ acts on a graph $\Gamma$, then we have a surjective *-homomorphism $\phi: C(\QBan(\Gamma)) \to C(G)$, $u \mapsto v$. 
\end{rem}

The following theorem shows that commutation with the magic unitary $u$ yields invariant subspaces. 

\begin{thm}[Theorem 2.3 of \cite{QBan}]\label{preserve}
Let $\alpha: C(X_n) \to C(X_n) \otimes C(G), \alpha(e_i) = \sum_{j} e_j \otimes v_{ji}$ be an action, where $G$ is a compact matrix quantum group and let $K$ be a linear subspace of $C(X_n)$. The matrix $(v_{ij})$ commutes with the projection onto $K$ if and only if $\alpha(K) \subset K \otimes C(G)$. 
\end{thm}

Looking at the spectral decomposition of the adjacency matrix, we see that this action preserves the eigenspaces of the adjacency matrix. 

\begin{cor}\label{Eigen}
Let $\Gamma=(V,E)$ be an undirected finite graph with adjacency matrix $\varepsilon$. The action $\alpha: C(V) \to C(V) \otimes C(\QBan(\Gamma))$, 
$\alpha(e_i) = \sum_{j} e_j \otimes u_{ji}$, 
preserves the eigenspaces of $\varepsilon$, i.e. $\alpha(E_\lambda) \subset E_\lambda \otimes C(\QBan(\Gamma))$ for all eigenspaces $E_\lambda$.
\end{cor}

\begin{proof}
It follows from the spectral decomposition that every projection $P_{E_{\lambda}}$ onto $E_{\lambda}$ is a polynomial in $\varepsilon$. Thus it commutes with the fundamental corepresentation $u$ and Theorem \ref{preserve} yields the assertion. 
\end{proof}

\subsection{Fourier transform}\label{Fourier}
One can obtain a $C^*$-algebra from the group $\Z_2^n$ by either considering the continuous functions $C(\Z_2^n)$ over the group or the group $C^*$-algebra $C^*(\Z_2^n)$. Since $\Z_2^n$ is abelian, we know that $C(\Z_2^n) \cong C^*(\Z_2^n)$ by Pontryagin duality. This isomorphism is given by the Fourier transform and its inverse. We write 
\begin{align*}
\Z_2^n&= \{t_1^{i_1} \dots t_n^{i_n} | i_1, \dots ,i_n \in \{0,1\}\},\\
C(\Z_2^n) &= \mathrm{span}( e_{t_1^{i_1} \dots t_n^{i_n}} \, | \, t_1^{i_1} \dots t_n^{i_n} \in \Z_2^n),\\
C^*(\Z_2^n) &= C^*(t_1, \dots , t_n \, | \,t_i = t_i^*,  t_i^2 = 1, t_i t_j = t_j t_i),\\
\intertext{where} e_{t_1^{i_1} \dots t_n^{i_n}} : \Z_2^n &\to \C, \qquad e_{t_1^{i_1} \dots t_n^{i_n}}({t_1^{j_1} \dots t_n^{j_n}}) = \delta_{i_1 j_1} \dots \delta_{i_n j_n}.
\end{align*}

The proof of the following proposition can be found in \cite{quizzy} for example. 
\begin{prop}
The *-homomorphisms
\begin{align*}
\phi: C(\Z_2^n) &\to C^*(\Z_2^n), \qquad e_{t_1^{i_1}\dots t_n^{i_n}} \to \frac{1}{2^n} \sum_{j_1, \dots, j_n=0}^1 (-1)^{i_1 j_1 + \dots + i_n j_n} t_1^{j_1} \dots t_n^{j_n}
\intertext{and} 
\psi: C^*(\Z_2^n) &\to C(\Z_2^n), \qquad  t_1^{i_1} \dots t_n^{i_n}\to  \sum_{j_1, \dots, j_n=0}^1 (-1)^{i_1 j_1 + \dots + i_n j_n}e_{t_1^{j_1}\dots t_n^{j_n}}, 
\end{align*}
where $i_1, \dots, i_n \in \{0,1\}$, are inverse to each other. The map $\phi$ is called Fourier transform, the map $\psi$ is called inverse Fourier transform. 
\end{prop}

\section{A criterion for a graph to have quantum symmetry}

In this section, we show that a graph has quantum symmetry if the automorphism group of the graph contains a certain pair of permutations. For this we need the following definition. 

\begin{defn}
Let $V = \{1,\dots, r\}$. We say that two permutations $\sigma:V \to V$ and $\tau: V \to V$ are \emph{disjoint}, if $\sigma(i) \neq i$ implies $\tau(i) =i$ and vice versa, for all $i \in V$. 
\end{defn}

\begin{thm}\label{noncomm} 
Let $\Gamma=(V,E)$ be a finite graph without multiple edges. If there exist two non-trivial, disjoint automorphisms $\sigma, \tau \in \Aut(\Gamma)$, $\mathrm{ord}(\sigma) = n, \mathrm{ord}(\tau)=m$, then we get a surjective *-homomorphism $\phi: C(\QBan(\Gamma)) \to C^*(\Z_n * \Z_m)$. In particular, $\Gamma$ has quantum symmetry if $n,m \geq 2$.
\end{thm}

\begin{proof}
Let $\sigma, \tau \in \Aut(\Gamma)$ be non-trivial disjoint automorphisms with $\mathrm{ord}(\sigma) = n, \mathrm{ord}(\tau)=m$. Define 
\begin{align*}
A&:=C^*(p_1, \dots, p_n, q_1, \dots, q_m| p_k=p_k^* = p_k^2, q_l=q_l^*=q_l^2, \sum_{k=1}^n p_k =1= \sum_{l=1}^m q_l)\\
&\cong C^*(\Z_n * \Z_m).
\end{align*} 
We want to use the universal property to get a surjective *-homomorphism \linebreak$\phi: C(\QBan(\Gamma)) \to A$. This yields the non-commutativity of $\QBan(\Gamma)$, since $p_k,q_l$ do not have to commute.
In order to do so, define 
\begin{align*}
u' := \sum_{l=1}^m \tau^l \otimes q_l + \sum_{k=1}^n \sigma^k \otimes p_k - \mathrm{id}_{\mathrm{M_r}(\mathbb{C}) \otimes A} \in \mathrm{M_r}(\mathbb{C}) \otimes A \cong \mathrm{M_r}(A),
\end{align*}
where $\tau^l, \sigma^k$ denote the permutation matrices corresponding to $\tau^l, \sigma^k \in \Aut(\Gamma)$. This yields
\begin{align*}
u'_{ij} = \sum_{l=1}^m \delta_{j\tau^l(i)} \otimes q_l + \sum_{k=1}^n \delta_{j\sigma^k(i)} \otimes p_k - \delta_{ij} \in \C \otimes A \cong A.
\end{align*}
Now, we show that $u'$ does fulfill the relations of $u \in \mathrm{M_r}(\mathbb{C}) \otimes A$, the fundamental representation of $\QBan(\Gamma)$. Since we have $\tau^l, \sigma^k \in \Aut(\Gamma)$, it holds $\tau^l \varepsilon = \varepsilon \tau^l$ and $\sigma^k \varepsilon = \varepsilon \sigma^k$ for all $1 \leq l \leq m$, $1 \leq k \leq n$, where $\varepsilon$ denotes the adjacency matrix of $\Gamma$. Therefore, we have\allowdisplaybreaks
\begin{align*}
u' (\varepsilon \otimes 1) &= \left(\sum_{l=1}^m \tau^l \otimes q_l + \sum_{k=1}^n \sigma^k \otimes p_k - \mathrm{id}_{\mathrm{M_r}(\mathbb{C}) \otimes A}\right) (\varepsilon \otimes 1)\\
&=\sum_{l=1}^m  \tau^l \varepsilon \otimes q_l + \sum_{k=1}^n  \sigma^k \varepsilon \otimes p_k - (\varepsilon \otimes 1)\\
&=\sum_{l=1}^m \varepsilon \tau^l \otimes q_l + \sum_{k=1}^n \varepsilon \sigma^k \otimes p_k - (\varepsilon \otimes 1)\\
&= (\varepsilon \otimes 1) \left(\sum_{l=1}^m \tau^l \otimes q_l + \sum_{k=1}^n \sigma^k \otimes p_k - \mathrm{id}_{\mathrm{M_r}(\mathbb{C}) \otimes A}\right)\\
&=(\varepsilon \otimes 1) u'.
\end{align*}
Furthermore, it holds
\begin{align*}
\sum_{i=1}^r u'_{ji} &=\sum_{i=1}^r \left(\sum_{l=1}^m \delta_{i\tau^l(j)} \otimes q_l + \sum_{k=1}^n \delta_{i\sigma^k(j)} \otimes p_k\right) - 1\otimes 1\\
&= 1 \otimes \left(\sum_{l=1}^m q_l \right)+ 1 \otimes \left(\sum_{k=1}^n p_k\right) - 1 \otimes 1\\
&= 1 \otimes 1.
\end{align*}
A similar computation shows $\sum_{i=1}^r u'_{ij} = 1 \otimes 1$. Since $\tau$ and $\sigma$ are disjoint, we have 
\begin{align*}
u'_{ij} = \sum_{l=1}^m \delta_{j\tau^l(i)} \otimes q_l + \sum_{k=1}^n \delta_{j\sigma^k(i)} \otimes p_k - \delta_{ij} =
\begin{cases}
\sum_{k \in N_{ij}} p_k, \text{ if } \sigma(i) \neq i\\
\sum_{l \in M_{ij}} q_l, \text{ if } \tau(i) \neq i\\
\delta_{ij}, \text{ otherwise,}
\end{cases}
\end{align*}
where $N_{ij} = \{ k \in \{1, \dots, n\}; \, \sigma^k(i)=j\}, M_{ij} = \{l \in \{1, \dots, m\}; \, \tau^l(i) =j\}$. Thus, all entries of $u'$ are projections. By the universal property, we get a *-homomorphism $\phi: C(\QBan(\Gamma)) \to A, u \mapsto u'.$

It remains to show that $\phi$ is surjective. We know $\mathrm{ord}(\sigma)=n$. By decomposing $\sigma$ into a product of disjoint cycles, we see that there exist $s_1, \dots, s_a \in V$ such that for all $k_1\neq k_2, k_1, k_2 \in \{1, \dots ,n\}$, we have 
\begin{align*}
(\sigma^{k_1}(s_1), \dots, \sigma^{k_1}(s_a)) \neq (\sigma^{k_2}(s_1), \dots, \sigma^{k_2}(s_a)).
\end{align*}
By similar considerations, there exist $t_1, \dots, t_b \in V$ such that 
\begin{align*}
(\tau^{l_1}(t_1), \dots, \tau^{l_1}(t_b)) \neq (\tau^{l_2}(t_1), \dots, \tau^{l_2}(t_b))
\end{align*}
for $l_1\neq l_2, l_1,l_2 \in \{1,\dots,m\}$. Therefore, we have 
\begin{align*}
\phi(u_{s_1\sigma^k(s_1)}\dots u_{s_a\sigma^k(s_a)}) &=u'_{s_1\sigma^k(s_1)}\dots u'_{s_a\sigma^k(s_a)} = p_k,\\
\phi(u_{t_1\tau^l(t_1)}\dots u_{t_b\tau^l(t_b)}) &=u'_{t_1\tau^l(t_1)}\dots u'_{t_b\tau^l(t_b)} = q_l
\end{align*}
 for all $k \in \{1 ,\dots ,n\}, l \in \{1 ,\dots, m\}$ and since $A$ is generated by $p_k$ and $q_l$, $\phi$ is surjective. 
\end{proof}

\begin{rem}
Let $K_4$ be the full graph on 4 points. We know that $\Aut(K_4) = S_4$ and $\QBan(K_4)=S_4^+$. We have disjoint automorphisms in $S_4$, where for example $\sigma=(12), \tau=(34) \in S_4$ give us the well known surjective *-homomorphism 
\begin{align*}
\varphi:C(S_4^+) &\to C^*(p,q \, | \, p=p^*=p^2, q = q^* =q^2), \\
u &\mapsto \begin{pmatrix} p&1-p&0&0\\1-p&p&0&0\\0&0&q&1-q\\0&0&1-q&q\end{pmatrix}, 
\end{align*}
yielding the non-commutativity of $S_4^+$. 
\end{rem}

\begin{rem}
Let $\Gamma=(V,E)$ be a finite graph without multiple edges, where there exist two non-trivial, disjoint automorphisms $\sigma, \tau \in \Aut(\Gamma)$. To show that $\Gamma$ has quantum symmetry it is enough to see that we have the surjective *-homomorphism 
\begin{align*}
\phi': C(\QBan(\Gamma)) &\to C^*(p,q \, | \,p=p^*=p^2, q = q^* =q^2), \\
u &\mapsto \sigma \otimes p + \tau \otimes q + \mathrm{id}_{M_r(\C)} \otimes (1-q-p).
\end{align*} 
\end{rem}

\begin{rem}
At the moment, we do not have an example of a graph $\Gamma$, where $\Aut(\Gamma)$ does not contain two disjoint automorphisms but the graph has quantum symmetry. 
\end{rem}

\section{The Clebsch graph has quantum symmetry}

As an application of Theorem \ref{noncomm}, we show that the Clebsch graph does have quantum symmetry. In Section \ref{sect}, we will study the quantum automorphism group of this graph. 

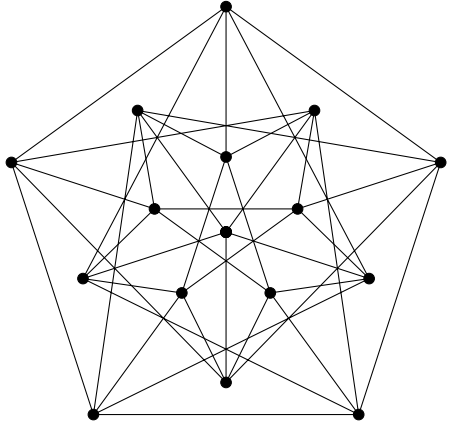
\begin{figure}[H]
\begin{center}
\begin{tikzpicture}
\draw (18:1cm) -- (162:1cm) -- (306:1cm) -- (90:1cm) -- (234:1cm) -- cycle;
\draw (18:3cm) -- (90:3cm) -- (162:3cm) -- (234:3cm) -- (306:3cm) -- cycle;
\draw (18:1cm) -- (54:2cm) --(90:1cm) -- (126:2cm) -- (162:1cm) -- (198:2cm) -- (234:1cm) -- (270:2cm) -- (306:1cm) -- (342:2cm) -- cycle;
\draw (18:3cm) -- (126:2cm) -- (234:3cm)-- (342:2cm) -- (90:3cm) -- (198:2cm) -- (306:3cm) -- (54:2cm) -- (162:3cm) -- (270:2cm) -- cycle;
\foreach \x in {18,90,162,234,306}
{\draw (\x:1cm) -- (\x:3cm);
\draw[black,fill=black] (\x:3cm) circle (2pt);
\draw[black,fill=black] (\x:1cm) circle (2pt);}
\foreach \x in {54, 126, 198, 270, 342}
{\draw[black,fill=black] (\x:2cm) circle (2pt);
\draw[black,fill=black] (\x:0cm) circle (2pt);
\draw (\x:0cm) -- (\x:2cm);}
\end{tikzpicture}
\end{center}
\caption{\label{figure}The Clebsch graph}
\end{figure}

\begin{prop}
The Clebsch graph has disjoint automorphisms.
\end{prop}

\begin{proof}
We label the graph as follows
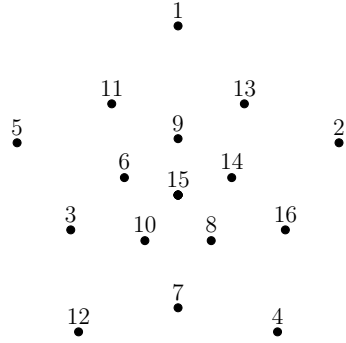
\begin{figure}[H]
\begin{center}
\begin{tikzpicture}[scale=0.75, transform shape]
\foreach \x in {18,90,162,234,306}
{\draw[black,fill=black] (\x:3cm) circle (2pt);
\draw[black,fill=black] (\x:1cm) circle (2pt);}
\foreach \x in {54, 126, 198, 270, 342}
{\draw[black,fill=black] (\x:2cm) circle (2pt);
\draw[black,fill=black] (\x:0cm) circle (2pt);}
\coordinate[label=above:$1$] (1) at (90:3cm);
\coordinate[label=above:$2$] (2) at (18:3cm);
\coordinate[label=above:$3$] (3) at (198:2cm);
\coordinate[label=above:$4$] (4) at (306:3cm);
\coordinate[label=above:$5$] (5) at (162:3cm);
\coordinate[label=above:$6$] (6) at (162:1cm);
\coordinate[label=above:$7$] (7) at (270:2cm);
\coordinate[label=above:$8$] (8) at (306:1cm);
\coordinate[label=above:$9$] (9) at (90:1cm);
\coordinate[label=above:$10$] (10) at (234:1cm);
\coordinate[label=above:$11$] (11) at (126:2cm);
\coordinate[label=above:$12$] (12) at (234:3cm);
\coordinate[label=above:$13$] (13) at (54:2cm);
\coordinate[label=above:$14$] (14) at (18:1cm);
\coordinate[label=above:$15$] (15) at (90:0cm);
\coordinate[label=above:$16$] (16) at (342:2cm);
\end{tikzpicture}
\end{center}
\end{figure}
Then we get two non-trivial disjoint automorphisms of this graph
\begin{align*}
\sigma &= (2\,3)(6\,7)(10\,11)(14\,15),\\
\tau &= (1\,4)(5\,8)(9\,12)(13\,16).
\end{align*}
\end{proof}

\begin{cor}\label{C}
The Clebsch graph does have quantum symmetries, \linebreak i.e. $C(G_{aut}^+(\Gamma_{Clebsch}))$ is non-commutative. 
\end{cor}

\begin{proof}
By Theorem \ref{noncomm}, we get that $C(G_{aut}^+(\Gamma_{Clebsch}))$ is non-commutative. Looking at the proof of Theorem \ref{noncomm}, we get the surjective *-homomorphism \linebreak$\varphi:C(G_{aut}^+(\Gamma_{Clebsch})) \to C^*(p,q \, | \, p = p^* = p^2, q = q^* = q^2)$, 
\begin{align*}
u &\mapsto u'=\begin{pmatrix} u''&0&0&0\\ 0&u''&0&0\\0&0&u''&0\\0&0&0&u''\end{pmatrix},
\intertext{where}
u'' &=\begin{pmatrix} q&0&0&1-q\\0&p&1-p&0\\0&1-p&p&0\\1-q&0&0&q\end{pmatrix}.
\end{align*}
\end{proof}

\begin{rem}~
\begin{itemize}
\item[(i)]The Clebsch graph is the folded $5$-cube graph, which will be introduced in Section \ref{sect}. There we will study the quantum automorphism group for $(2m+1)$-folded cube graphs going far beyond Corollary \ref{C}.
\item[(ii)] Using Theorem \ref{noncomm}, it is also easy to see that the folded cube graphs have quantum symmetry, but this will also follow from our main result (Theorem \ref{main}). 
\end{itemize}
\end{rem}

\section{The quantum group $SO_n^{-1}$}

Now, we will have a closer look at the quantum group $SO_n^{-1}$, but first we define $O_n^{-1}$, which appeared in \cite{hyperoctahedral} as the quantum automorphism group of the hypercube graph. For both it is immediate to check that the comultiplication $\Delta$ is a *-homomorphism. 

\begin{defn}
We define $O_n^{-1}$ to be the compact matrix quantum \linebreak group $(C(O_n^{-1}), u)$, where $C(O_n^{-1})$ is the universal $C^*$-algebra with generators $u_{ij}$, $1 \leq i,j \leq n$ and relations
\begin{align}
&u_{ij} = u_{ij}^*, &&1\leq i,j \leq n, \label{7.1}\\
&\sum_{k=1}^n u_{ik}u_{jk} = \sum_{k=1}^n u_{ki}u_{kj} = \delta_{ij}, && 1\leq i,j \leq n,\label{7.2}\\
&u_{ij}u_{ik} = - u_{ik}u_{ij}, u_{ji}u_{ki} = -u_{ki}u_{ji}, &&k\neq j,\label{7.3}\\
&u_{ij}u_{kl}=u_{kl}u_{ij}, &&i\neq k, j \neq l.\label{7.4}
\end{align}
\end{defn}

For $n=3$, $SO_n^{-1}$ appeared in \cite{4points}, where Banica and Bichon showed $SO_3^{-1} = S_4^+$. Our main result in this paper is that for $n$ odd, $SO_n^{-1}$ is the quantum automorphism group of the folded $n$-cube graph. 

\begin{defn}
We define $SO_n^{-1}$ to be the compact matrix quantum \linebreak group $(C(SO_n^{-1}), u)$, where $C(SO_n^{-1})$ is the universal $C^*$-algebra with generators $u_{ij}$, $1 \leq i,j \leq n$, Relations \eqref{7.1} -- \eqref{7.4} and 
\begin{align}
\sum_{\sigma \in S_n} u_{\sigma(1)1}\dots u_{\sigma(n)n} =1.\label{7.5}
\end{align}
\end{defn}

\begin{lem}\label{sumzero}
Let $(u_{ij})_{1 \leq i,j \leq n}$ be the generators of $C(SO_n^{-1})$. Then 
\begin{align*}
\sum_{\sigma \in S_n} u_{\sigma(1)1} \dots u_{\sigma(n-1)n-1}u_{\sigma(n)k} =0
\end{align*}
for $k \neq n$. 
\end{lem}

\begin{proof}
Let $1 \leq k \leq n-1$. Using Relations \eqref{7.3} and \eqref{7.4} we get
\begin{align*}
u_{\sigma(1)1} \dots u_{\sigma(k)k} \dots u_{\sigma(n-1)n-1}u_{\sigma(n)k} &= - u_{\sigma(1)1} \dots u_{\sigma(n)k} \dots  u_{\sigma(n-1)n-1}u_{\sigma(k)k}\\
&=-u_{\tau(1)1} \dots u_{\tau(k)k} \dots u_{\tau(n-1)n-1}u_{\tau(n)k}
\end{align*} 
for $\tau = \sigma \circ(k \, n) \in S_n$. Therefore the summands corresponding to $\sigma$ and $\tau$ sum up to zero. The result is then clear.
\end{proof}

The next lemma gives an equivalent formulation of Relation \eqref{7.5}. One direction is a special case of \cite[Lemma 4.6]{Dualpairs}.

\begin{lem}\label{SO} Let $A$ be a $C^*$-algebra and let $u_{ij} \in A$ be elements that fulfill Relations $\eqref{7.1}-\eqref{7.4}$. Let $j \in \{1,\dots,n\}$ and define
\begin{align*}
I_j = \{ (i_1,\dots ,i_{n-1}) \in \{1,\dots ,n\}^{n-1} \, | \, i_a \neq i_b \text{ for } a \neq b, i_s \neq j \text{ for all } s\}.
\end{align*}
The following are equivalent
\begin{itemize}
\item[(i)] We have \begin{align*}1&=\sum_{\sigma \in S_n} u_{\sigma(1)1}\dots u_{\sigma(n)n}.\intertext{
\item[(ii)] It holds} u_{jn} &= \sum_{(i_1,\dots,i_{n-1}) \in I_j} u_{i_1 1}\dots u_{i_{n-1}n-1}, \quad 1\leq j\leq n.
 \end{align*}
\end{itemize}
\end{lem}

\begin{proof}
We first show that (ii) implies (i). It holds
\begin{align*}
1 = \sum_{j=1}^n u_{jn}^2 = \sum_{j=1}^n \sum_{(i_1,\dots,i_{n-1}) \in I_j} u_{i_1 1}\dots u_{i_{n-1}n-1}u_{jn}, 
\end{align*}
where we used Relation \eqref{7.2} and (ii). Furthermore, we have 
\begin{align*}
\sum_{j=1}^n \sum_{(i_1,\dots,i_{n-1}) \in I_j} u_{i_1 1}\dots u_{i_{n-1}n-1} u_{jn}
&= \sum_{\substack{i_1, \dots, i_n;\\ i_a \neq i_b \text{ for } a\neq b}} u_{i_1 1} \dots u_{i_n n}\\
&= \sum_{\sigma \in S_n} u_{\sigma(1)1} \dots u_{\sigma(n)n}
\end{align*}
and thus (ii) implies (i). 

Now we show that (i) implies (ii). We have 
\begin{align*}
u_{jn} = \sum_{\sigma \in S_n} u_{\sigma(1)1} \dots u_{\sigma(n)n}u_{jn} = \sum_{k=1}^n \sum_{\sigma \in S_n} u_{\sigma(1)1} \dots u_{\sigma(n-1)n-1}u_{\sigma(n)k}u_{jk}, 
\end{align*}
since $\sum_{\sigma \in S_n} u_{\sigma(1)1} \dots u_{\sigma(n-1)n-1}u_{\sigma(n)k}u_{jk} =0$ for $k \neq n$ by Lemma \ref{sumzero}. We get
\begin{align*}
 \sum_{k=1}^n \sum_{\sigma \in S_n} u_{\sigma(1)1} \dots u_{\sigma(n-1)n-1} u_{\sigma(n)k}u_{jk} &= \sum_{\sigma \in S_n} u_{\sigma(1)1} \dots u_{\sigma(n-1)n-1}\sum_{k=1}^n u_{\sigma(n)k}u_{jk}\\
&= \sum_{\sigma \in S_n} u_{\sigma(1)1}\dots u_{\sigma(n-1)n-1}\delta_{\sigma(n)j}\\
&= \sum_{(i_1,\dots,i_{n-1}) \in I_j} u_{i_1 1}\dots u_{i_{n-1}n-1}, 
\end{align*}
where we used Relation \eqref{7.2} and we obtain $u_{jn} = \sum_{(i_1,\dots,i_{n-1}) \in I_j} u_{i_1 1}\dots u_{i_{n-1}n-1}$. 
\end{proof}

We now discuss representations of $SO_{2m+1}^{-1}$. For definitions and background for this proposition, we refer to \cite{hopfgalois, subgroups, schauenburg}.

\begin{prop}
The category of corepresentations of $SO_{2m+1}^{-1}$ is tensor equivalent to the category of representations of $SO_{2m+1}$. 
\end{prop}

\begin{proof}
We first show that $C(SO_{2m+1}^{-1})$ is a cocycle twist of $C(SO_{2m+1})$ by proceeding like in \cite[Section 4]{subgroups}. Take the unique bicharacter $\sigma: \Z_2^{2m} \times \Z_2^{2m} \to \{\pm 1\}$ with
\begin{align*}
&\sigma(t_i, t_j) = - 1 = - \sigma(t_j, t_i) ,&& \text{for } 1 \leq i < j \leq 2m,\\
&\sigma(t_i, t_i) = (-1)^m, && \text{for } 1 \leq i \leq 2m+1,\\
&\sigma(t_i, t_{2m+1}) = (-1)^{m-i} = - \sigma(t_{2m+1}, t_i), && \text{for } 1 \leq i \leq 2m,
\end{align*}
where we use the identification $\Z_2^{2m} = \langle t_1, \dots, t_{2m+1} \, | \, t_i^2=1, t_i t_j = t_j t_i, t_{2m+1} = t_1\dots t_{2m}\rangle$.
Let $H$ be the subgroup of diagonal matrices in $SO_{2m+1}$ having $\pm 1$ entries. We get a surjective *-homomorphism 
\begin{align*}
\pi : C(SO_{2m+1}) &\to C^*(\Z_2^{2m})\\
u_{ij} &\mapsto \delta_{ij} t_i
\end{align*}
by restricting the functions on $SO_{2m+1}$ to $H$ and using Fourier transform. Thus we can form the twisted algebra $C(SO_{2m+1})^{\sigma}$, 
where we have the multiplication
\begin{align*}
[u_{ij}] [u_{kl}] = \sigma(t_i, t_k) \sigma^{-1}(t_j, t_l) [u_{ij}u_{kl}] =\sigma(t_i, t_k) \sigma(t_j, t_l) [u_{ij}u_{kl}].
\end{align*}
We see that the generators $[u_{ij}]$ of $C(SO_{2m+1})^{\sigma}$ fulfill the same relations as the generators of $C(SO_{2m+1}^{-1})$ and thus we get an surjective *-homomorphism \linebreak$\phi: C(SO_{2m+1}^{-1}) \to C(SO_{2m+1})^{\sigma}, \,u_{ij} \mapsto [u_{ij}].$ This is an isomorphism for example by using Theorem 3.5 of \cite{Kassel}.
 Now, Corollary 1.4 and Proposition 2.1 of \cite{hopfgalois} yield the assertion.
\end{proof}

\section{Quantum automorphism groups of folded cube graphs}\label{sect}

In what follows, we will introduce folded cube graphs $FQ_n$ and show that for odd $n$, the quantum automorphism group of $FQ_n$ is $SO_n^{-1}$.

\subsection{The folded $n$-cube graph $FQ_n$}

\begin{defn}
The folded $n$-cube graph $FQ_n$ is the graph with vertex set $V= \{ (x_1, \dots, x_{n-1})  \, | \, x_i \in \{0,1\}\}$, where two vertices $(x_1,\dots, x_{n-1})$ and $ (y_1, \dots,y_{n-1})$ are connected if they differ at exactly one position or if $(y_1, \dots,y_{n-1}) = (1-x_1, \dots, 1-x_{n-1})$. 
\end{defn}

\begin{rem}
To justify the name, one can obtain the folded $n$-cube graph by identifing every opposite pair of vertices from the $n$-hypercube graph. 
\end{rem}

\subsection{The folded cube graph as Cayley graph}
It is known that the folded cube graphs are Cayley graphs, we recall this fact in the next lemma. 

\begin{lem}
The folded $n$-cube graph $FQ_n$ is the Cayley graph of the group \linebreak$\Z_2^{n-1}= \langle t_1, \dots t_n \rangle$, where the generators $t_i$ fulfill the relations $t_i^2=1, t_i t_j = t_j t_i, t_n = t_1\dots t_{n-1}$. 
\end{lem}

\begin{proof}
Consider the Cayley graph of $\Z_2^{n-1}= \langle t_1, \dots t_n \rangle$. The vertices are elements of $\Z_2^{n-1}$, which are products of the form $g = t_1^{i_1}\dots t_{n-1}^{i_{n-1}}$. The exponents are in one to one correspondence to $(x_1, \dots, x_{n-1}), \, x_i \in \{0,1\}$, thus the vertices of the Cayley graph are the vertices of the folded $n$-cube graph. The edges of the Cayley graph are drawn between vertices $g,h$, where $g = h t_i$ for some $i$. For $k \in \{1,\dots, n-1\}$, the operation $h \to h t_k$ changes the $k$-th exponent to $1-i_k$, so we get edges between vertices that differ at exactly one expontent. The operation $h \to h t_n$ takes $t_1^{j_1}\dots t_{n-1}^{j_{n-1}}$ to $t_1^{1-j_1}\dots t_{n-1}^{1-j_{n-1}}$, thus we get the remaining edges of $FQ_n$. 
\end{proof}

\subsection{Eigenvalues and Eigenvectors of $FQ_n$}
We will now discuss the eigenvalues and eigenvectors of the adjacency matrix of $FQ_n$. 

\begin{lem}\label{EV}
The eigenvectors and corresponding eigenvalues of $FQ_n$ are given by 
\begin{align*}
w_{i_1\dots i_{n-1}}&= \sum_{j_1,\dots,j_{n-1}=0}^1(-1)^{i_1j_1 + \dots +i_{n-1}j_{n-1}}e_{t_1^{j_1}\dots t_{n-1}^{j_{n-1}}}\\
\lambda_{i_1\dots i_{n-1}} &= (-1)^{i_1}+ \ldots +(-1)^{i_{n-1}}+(-1)^{i_1+\ldots+i_{n-1}},
\end{align*}
when the vector space spanned by the vertices of $FQ_n$ is identified with $C(\Z_2^{n-1})$. 
\end{lem}

\begin{proof}
Let $\varepsilon$ be the adjacency matrix of $FQ_n$. Then we know for a vertex $p$ and a function $f$ on the vertices that
\begin{align*}
\varepsilon f(p) = \sum_{q ; (q,p) \in E} f(q).
\end{align*}
This yields
\begin{align*}
\varepsilon  e_{t_1^{j_1}\dots t_{n-1}^{j_{n-1}}} = \sum_{k=1}^n e_{t_kt_1^{j_1}\dots t_{n-1}^{j_{n-1}}} = e_{t_1^{j_1 +1}\dots t_{n-1}^{j_{n-1}}} + \dots+ e_{t_1^{j_1}\dots t_{n-1}^{j_{n-1}+1}} + e_{t_1^{j_1+1}\dots t_{n-1}^{j_{n-1}+1}}.
\end{align*}
For the vectors in the statement we get\allowdisplaybreaks
\begin{align*}
\varepsilon w_{i_1 \dots i_{n-1}} &= \sum_{j_1,\dots,j_{n-1}} (-1)^{i_1j_1 + \dots +i_{n-1}j_{n-1}}\varepsilon e_{t_1^{j_1}\dots t_{n-1}^{j_{n-1}}}\\
&=\sum_{s=1}^{n-1} \sum_{j_1,\dots,j_{n-1}} (-1)^{i_1j_1 + \dots +i_{n-1}j_{n-1}}e_{t_1^{j_1}\dots t_s^{j_s+1}\dots t_{n-1}^{j_{n-1}}}\\
&\qquad + \sum_{j_1,\dots,j_{n-1}} (-1)^{i_1j_1 + \dots +i_{n-1}j_{n-1}}e_{t_1^{j_1+1}\dots t_{n-1}^{j_{n-1}+1}}.
\intertext{Using the index shift $j_s' = j_s +1\bmod 2$, for $s \in \{1, \dots n-1\}$, we get}
\varepsilon w_{i_1 \dots i_{n-1}}&=\sum_{s=1}^{n-1} \sum_{j_1,\dots j_s', \dots j_{n-1}} (-1)^{i_1j_1 + \dots + i_s(j_s'+1)+ \dots +i_{n-1}j_{n-1}}e_{t_1^{j_1}\dots t_s^{j_s'}\dots t_{n-1}^{j_{n-1}}}\\
&\qquad + \sum_{j_1',\dots,j_{n-1}'} (-1)^{i_1(j_1'+1) + \dots +i_{n-1}(j_{n-1}'+1)}e_{t_1^{j_1'}\dots t_{n-1}^{j_{n-1}'}}\\
&=\sum_{s=1}^{n-1} \sum_{j_1,\dots, j_s', \dots, j_{n-1}}(-1)^{i_s}(-1)^{i_1j_1 + \dots +i_{n-1}j_{n-1}}e_{t_1^{j_1}\dots t_s^{j_s'}\dots t_{n-1}^{j_{n-1}}} \\
&\qquad + \sum_{j_1',\dots,j_{n-1}'} (-1)^{i_1+ \dots + i_{n-1}}(-1)^{i_1j_1' + \dots +i_{n-1}j_{n-1}'}e_{t_1^{j_1'}\dots t_{n-1}^{j_{n-1}'}}\\
&= ((-1)^{i_1}+ \ldots +(-1)^{i_{n-1}}+(-1)^{i_1+\ldots+i_{n-1}})w_{i_1\dots i_{n-1}}\\
&= \lambda_{i_1\dots i_{n-1}}w_{i_1\dots i_{n-1}}.
\end{align*}
Since those are $2^{n-1}$ vectors that are linearly independent, the assertion follows. 
\end{proof}

The following lemma shows what the eigenvectors look like if we identify the vector space spanned by the vertices of $FQ_n$ with $C^*(\Z_2^{n-1})$.

\begin{lem}\label{ev}
In $C^*(\Z_2^{n-1})= C^*(t_1, \dots, t_n \, | \, t_i^2=1, t_i t_j = t_j t_i, t_n = t_1\dots t_{n-1})$ the eigenvectors of $FQ_n$ are 
\begin{align*}
\hat{w}_{i_1\dots i_{n-1}} = t_1^{i_1} \dots t_{n-1}^{i_{n-1}}
\end{align*}
corresponding to the eigenvalues $\lambda_{i_1\dots i_{n-1}}$ from Lemma \ref{EV}.
\end{lem}

\begin{proof}
We obtain $\hat{w}_{i_1\dots i_{n-1}} $ by using the Fourier transform (see Section \ref{Fourier}) on $w_{i_1\dots i_{n-1}}$ from Lemma \ref{EV}. 
\end{proof}

Note that certain eigenvalues in Lemma \ref{EV} coincide. We get a better description of the eigenvalues and eigenspaces of $FQ_n$ in the next lemma. 

\begin{lem}\label{ES}
The eigenvalues of $FQ_n$ are given by $\lambda_k = n-2k$ for $k\in 2\Z\cap \{0,\dots,n\}$. The eigenvectors $t_1^{i_1} \dots t_{n-1}^{i_{n-1}}$ corresponding to $\lambda_k$ have word lengths $k$ or $k-1$ and form a basis of $E_{\lambda_k}$. Here $E_{\lambda_k}$ denotes the eigenspace to the eigenvalue $\lambda_k$.
\end{lem}

\begin{proof}
Let $k\in 2\Z\cap \{0,\dots,n\}$. By Lemma \ref{EV} and Lemma \ref{ev}, we get that an eigenvector $t_1^{i_1} \dots t_{n-1}^{i_{n-1}}$ of word length $k$ (here $k\neq n$, if $n$ is even) with respect to $t_1, \dots, t_{n-1}$ corresponds to the eigenvalue
\begin{align*}
(-1)^{i_1}+ \ldots +(-1)^{i_{n-1}}+(-1)^{i_1+\ldots+i_{n-1}} = -k + (n-1-k) + 1 = n-2k.
\end{align*}
Now consider an eigenvector $t_1^{i_1} \dots t_{n-1}^{i_{n-1}}$ of word length $k-1$. Then we get the eigenvalue
\begin{align*}
(-1)^{i_1}+ \ldots +(-1)^{i_{n-1}}+(-1)^{i_1+\ldots+i_{n-1}} = -(k-1) + (n-k) -1 = n-2k.
\end{align*}
We go through all the eigenvectors of Lemma \ref{ev} in this way and we obtain exactly the eigenvalues $\lambda_k = n-2k$. Since the eigenvectors of word lengths $k$ or $k-1$ are exactly those corresponding to $\lambda_k$, they form a basis of $E_{\lambda_k}$.
\end{proof}

\subsection{The quantum automorphism group of $FQ_{2m+1}$}
For the rest of this section, we restrict to the folded $n$-cube graphs, where $n =2m +1$ is odd. We show that in this case, the quantum automorphism group is $SO_n^{-1}$. We need the following lemma. 

\begin{lem}\label{P}
Let $\tau_1, \dots ,\tau_n$ be generators of $C^*(\Z_2^{n-1})$ with $\tau_i^2=1, \tau_i \tau_j = \tau_j \tau_i, \tau_n = \tau_1\dots \tau_{n-1}$ and let A be a $C^*$-algebra with elements $u_{ij} \in A$ fulfilling Relations \eqref{7.1} -- \eqref{7.4}. Let $(i_1, \dots, i_l) \in \{ 1,\dots, n\}^l$ with $i_a \neq i_b$ for $a \neq b$, where $1 \leq l \leq n$. Then
\begin{align*}
\sum_{j_1, \dots, j_l=1}^n \tau_{j_1}\dots\tau_{j_l} \otimes u_{j_1i_1} \dots u_{j_l i_l} = \sum_{\substack{j_1, \dots, j_{l};\\j_a\neq j_b \text{ for } a\neq b}} \tau_{j_1}\dots\tau_{j_{l}}\otimes u_{j_1i_1}\dots u_{j_{l}i_l}.
\end{align*} 
\end{lem}

\begin{proof}
Let $j_{s} = j_{s+1} =k$ and let the remaining $j_l$ be arbitrary. Summing over $k$, we get
\begin{align*}
\sum_{k=1}^n \tau_{j_1} \dots \tau_{j_{s-1}} &\tau_k^2 \tau_{j_{s+2}}\dots \tau_{j_l} \otimes u_{j_1i_1}\dots u_{ki_s}u_{ki_{s+1}}\dots u_{j_li_l} \\
&= \tau_{j_1} \dots \tau_{j_{s-1}}\tau_{j_{s+2}}\dots \tau_{j_l}\otimes u_{j_1i_1}\dots \left(\sum_{k=1}^n u_{ki_s}u_{ki_{s+1}}\right)\dots u_{j_li_l}\\
&=0
\end{align*}
by Relation \eqref{7.2} since $i_s \neq i_{s+1}$. Doing this for all $s \in \{1, \dots, l-1\}$ we get 
\begin{align*}
\sum_{j_1, \dots, j_l=1}^n \tau_{j_1}\dots\tau_{j_l} \otimes u_{j_1i_1} \dots u_{j_l i_l} =\sum_{j_1\neq\dots \neq j_l} \tau_{j_1}\dots\tau_{j_l} \otimes u_{j_1i_1} \dots u_{j_l i_l}.
\end{align*}
Now, let $j_s = j_{s+2} = k$ and let $j_1 \neq \dots \neq j_l$. Since $k=j_s \neq j_{s+1}$ and $i_a \neq i_b$ for $a \neq b$, we have $u_{k i_s} u_{j_{s+1}i_{s+1}} = u_{j_{s+1}i_{s+1}}u_{k i_s}$. We also know that $\tau_k \tau_{j_{s+1}} = \tau_{j_{s+1}} \tau_k$ and thus
\begin{align*}
\sum_{k=1}^n &\tau_{j_1} \dots \tau_{j_{s-1}} \tau_k \tau_{j_{s+1}} \tau_k \tau_{j_{s+3}}\dots \tau_{j_l} \otimes u_{j_1i_1}\dots u_{ki_s}u_{j_{s+1}i_{s+1}}u_{ki_{s+2}}\dots u_{j_li_l}\\
&= \tau_{j_1} \dots \tau_{j_{s-1}} \tau_{j_{s+1}} \tau_{j_{s+3}}\dots \tau_{j_l} \otimes u_{j_1i_1}\dots u_{j_{s+1}i_{s+1}}\left(\sum_{k=1}^nu_{ki_s}u_{ki_{s+2}}\right)\dots u_{j_li_l}\\
&=0
\end{align*}
by Relation \eqref{7.2} since $i_s \neq i_{s+2}$. This yields
\begin{align*}
\sum_{j_1, \dots, j_l=1}^n \tau_{j_1}\dots\tau_{j_l} \otimes u_{j_1i_1} \dots u_{j_l i_l}=\sum_{\substack{j_1, \dots, j_{l};\\j_a\neq j_b \text{ for } 0< |a-b| \leq 2}} \tau_{j_1}\dots\tau_{j_{l}}\otimes u_{j_1i_1}\dots u_{j_{l}i_l}
\end{align*}
The assertion follows after iterating this argument $l$ times.
\end{proof}

We first show that $SO_n^{-1}$ acts on the folded $n$-cube graph. 

\begin{lem}\label{act}
For $n$ odd, the quantum group $SO_n^{-1}$ acts on $FQ_n$. 
\end{lem}

\begin{proof}
We need to show that there exists an action 
\begin{align*}
\alpha: C(V_{FQ_n}) \to C(V_{FQ_n}) \otimes C(SO_n^{-1}), \qquad \alpha(e_i) = \sum_{j=1}^{|V_{FQ_n}|} e_j \otimes v_{ji}
\end{align*}
 such that $(v_{ij})$ commutes with the adjacency matrix of $FQ_n$. By Fourier transform, this is the same as getting an action 
\begin{align*}
\alpha: C^*(\Z_2^{n-1}) \to C^*(\Z_2^{n-1}) \otimes C(SO_n^{-1}),
\end{align*}
where we identify the functions on the vertex set of $FQ_n$ with $C^*(\Z_2^{n-1})$. 
We claim that 
\begin{align*}
\alpha(\tau_i) = \sum_{j=1}^n \tau_j \otimes u_{ji}
\end{align*}
gives the answer, where 
\begin{align*}
\tau_i = t_1\dots \check{t_i}\dots t_{n-1} \text{ for } 1 \leq i \leq n-1, \qquad \tau_n = t_n
\end{align*}
 for $t_i$ as in Lemma \ref{ev} and $(u_{ij})$ is the fundamental corepresentation of $SO_n^{-1}$. Here $\check{t_i}$ means that $t_i$ is not part of the product. These $\tau_i$ generate $C^*(\Z_2^{n-1})$, with relations $\tau_i = \tau_i^*, \tau_i^2=1, \tau_i\tau_j = \tau_j\tau_i$ and $\tau_n = \tau_1\dots\tau_{n-1}$.
Define 
\begin{align*}
\tau_i' = \sum_{j=1}^n \tau_j \otimes u_{ji}.
\end{align*}

To show that $\alpha$ defines a *-homomorphism, we have to show that the relations of the generators $\tau_i$ also hold for $\tau_i'$. It is obvious that $(\tau_i')^* = \tau_i'$. Using Relations \eqref{7.2}--\eqref{7.4} it is straightforward to check that $(\tau'_i)^2=1$ and $\tau_i'\tau_j'=\tau_j'\tau_i'$. Now, we show $\tau'_n = \tau'_1\dots\tau'_{n-1}$. By Lemma \ref{P}, it holds
\begin{align*}
\tau'_1\dots\tau'_{n-1} &= \sum_{\substack{i_1, \dots, i_{n-1};\\i_a\neq i_b \text{ for } a\neq b}} \tau_{i_1}\dots\tau_{i_{n-1}}\otimes u_{i_11}\dots u_{i_{n-1}n-1}\\
&= \sum_{j=1}^n \sum_{(i_1,\dots,i_{n-1}) \in I_j} \tau_{i_1}\dots\tau_{i_{n-1}}\otimes u_{i_11}\dots u_{i_{n-1}n-1},
\end{align*}
where $I_j =\{ (i_1,\dots, i_{n-1}) \in \{1,\dots, n\}^{n-1} \, | \, i_a \neq i_b \text{ for } a \neq b, i_s \neq j \text{ for all } s\}$ like in Lemma \ref{SO}. For all $(i_1,\dots,i_{n-1}) \in I_j$, we know that $\tau_{i_1}\dots\tau_{i_{n-1}} = \tau_1\dots \check{\tau_j}\dots \tau_n$. Using $\tau_n = \tau_1\dots\tau_{n-1}$ and $\tau_i^2=1$, we get $\tau_1\dots \check{\tau_j}\dots \tau_n= \tau_j$ and thus
\begin{align*}
 \sum_{j=1}^n \sum_{(i_1,\dots,i_{n-1}) \in I_j} &\tau_{i_1}\dots\tau_{i_{n-1}}\otimes u_{i_11}\dots u_{i_{n-1}n-1}\\
 &= \sum_{j=1}^n \left(\tau_j \otimes \sum_{(i_1,\dots,i_{n-1}) \in I_j}u_{i_11}\dots u_{i_{n-1}n-1}\right).
\end{align*}
The equivalent formulation of Relation \eqref{7.5} in Lemma \ref{SO} yields 
\begin{align*}
 \tau'_1\dots\tau'_{n-1} = \sum_{j=1}^n \left( \tau_j \otimes \sum_{(i_1,\dots,i_{n-1}) \in I_j}u_{i_11}\dots u_{i_{n-1}n-1} \right)
= \sum_{j=1}^n \tau_j \otimes u_{jn} = \tau'_n.
\end{align*}
Summarising, the map $\alpha$ exists and is a *-homomorphism. It is straightforward to check that $\alpha$ is unital and since $u$ is a corepresentation, $\alpha$ is coassociative.  

Now, we show that $\alpha(C^*(\Z_2^{n-1}))(1 \otimes C(SO_n^{-1}))$ is linearly dense in $C^*(\Z_2^{n-1}) \otimes C(SO_n^{-1})$. It holds
\begin{align*}
\sum_{i=1}^n \alpha(\tau_i)(1 \otimes u_{ki}) =  \sum_{j=1}^n \left( \tau_j \otimes \sum_{i=1}^n u_{ji}u_{ki} \right)= \sum_{j=1}^n \tau_j \otimes \delta_{jk} = \tau_k \otimes 1,
\end{align*}
thus $(\tau_k \otimes 1) \in \alpha(C^*(\Z_2^{n-1}))(1 \otimes C(SO_n^{-1}))$ for $1 \leq k \leq n$. Since $\alpha$ is unital, we also get $1 \otimes C(SO_n^{-1}) \subset \alpha(C^*(\Z_2^{n-1}))(1 \otimes C(SO_n^{-1}))$. By a standard argument, see for example \cite[Section 4.2]{SWe}, we get that $\alpha(C^*(\Z_2^{n-1}))(1 \otimes C(SO_n^{-1}))$ is linearly dense in $C^*(\Z_2^{n-1}) \otimes C(SO_n^{-1})$. 

It remains to show that the magic unitary matrix associated to $\alpha$ commutes with the adjacency matrix of $FQ_n$. 
We want to show that $\alpha$ preserves the eigenspaces of the adjacency matrix, i.e. $\alpha(E_\lambda) \subset E_\lambda \otimes C(SO_n^{-1})$ for all eigenspaces $E_\lambda$, then Theorem \ref{preserve} yields the assertion. Since it holds $t_j = \tau_j \tau_n$, by Lemma \ref{ev} we have eigenvectors
\begin{align*}
\hat{w}_{i_1 \dots i_{n-1}} = t_1^{i_1} \dots t_{n-1}^{i_{n-1}} = \begin{cases}  \tau_1^{i_1} \dots \tau_{n-1}^{i_{n-1}} &\text{ for } \sum_{k=1}^{n-1} i_k \text{ even} \\  \tau_1^{1-i_1}  \dots \tau_{n-1}^{1-i_{n-1}}&\text{ for } \sum_{k=1}^{n-1} i_k \text{ odd}\end{cases}
\end{align*}
corresponding to the eigenvalues $\lambda_{i_1 \dots i_{n-1}}$ as in Lemma \ref{EV}. Using Lemma \ref{ES}, we see that the eigenspaces $E_{\lambda_k}$ are spanned by eigenvectors $\tau_1^{i_1} \dots \tau_{n-1}^{i_{n-1}}$ of word lengths $k$ or $n-k$, where we consider the word length with respect to $\tau_1, \dots, \tau_{n-1}$. 

Let $1 \leq l \leq n-1$. By Lemma \ref{P}, we have for $i_1,\dots, i_l$, $i_a \neq i_b$ for $a \neq b$:
\begin{align*}
\alpha(\tau_{i_1} \dots \tau_{i_l}) = \sum_{\substack{j_1, \dots, j_{l};\\j_a\neq j_b \text{ for } a\neq b}} \tau_{j_1}\dots \tau_{j_l} \otimes u_{j_1 i_1} \dots u_{j_l i_l}.
\end{align*}
For $\tau_{j_1} \dots \tau_{j_l}$, where $j_s \neq n$ for all $s$, we immediately get that this is in the same eigenspace as $\tau_{i_1} \dots \tau_{i_l}$ since $\tau_{j_1} \dots \tau_{j_l}$ has the same word length as $\tau_{i_1} \dots \tau_{i_l}$. Take now $\tau_{j_1} \dots \tau_{j_l}$, where we have $j_s =n$ for some $s$. We get
\begin{align*}
\tau_{j_1}\dots \tau_{j_l} &= \tau_{j_1}\dots \check{\tau_{j_s}}\dots \tau_{j_l}\tau_n\\
&= \tau_{j_1}\dots \check{\tau_{j_s}}\dots \tau_{j_l}\tau_1 \dots \tau_{n-1}, 
\end{align*}
which has word length $n-1-(l-1) = n-l$, thus it is in the same eigenspace as $\tau_{i_1} \dots \tau_{i_l}$. 
This yields 
\begin{align*}
\alpha(E_\lambda) \subset E_\lambda \otimes C(SO_n^{-1}),
\end{align*}
for all eigenspaces $E_\lambda$ and thus $SO_n^{-1}$ acts on $FQ_n$ by Theorem \ref{preserve}.
\end{proof}

Now, we can prove our main theorem. 

\begin{thm}\label{main}
For $n$ odd, the quantum automorphism group of the folded $n$-cube graph $FQ_n$ is $SO_n^{-1}$. 
\end{thm}

\begin{proof}
By Lemma \ref{act} we get a surjective map $C(G_{aut}^+(FQ_n)) \to C(SO_n^{-1})$.
We have to show that this is an isomorphism between $C(SO_n^{-1})$ and \linebreak$C(G_{aut}^+(FQ_n))$. Consider the universal action on $FQ_n$
\begin{align*}
\beta: C^*(\Z_2^{n-1}) \to C^*(\Z_2^{n-1}) \otimes C(G_{aut}^+(FQ_n)).
\end{align*}
Consider $\tau_1, \dots ,\tau_n $ like in Lemma \ref{act}. They have word length $n-2$ or $n-1$ with respect to $t_1, \dots, t_{n-1}$ and they form a basis of $E_{-n+2}$ by Lemma \ref{ES}. Therefore, we get elements $x_{ij}$ such that 
\begin{align*}
\beta(\tau_i)= \sum_{j=1}^n \tau_j \otimes x_{ji}
\end{align*} 
by Corollary \ref{Eigen}. Similar to \cite{hyperoctahedral} one shows that $x_{ij}$ fulfill Relations \eqref{7.1}--\eqref{7.4}. It remains to show that Relation \eqref{7.5} holds. Applying $\beta$ to $\tau_n = \tau_1\dots \tau_{n-1}$ and using Lemma \ref{P} yields
\begin{align*}
\sum_{j} \tau_j \otimes x_{jn} = \beta(\tau_n)&= \sum_{\substack{i_1, \dots ,i_{n-1};\\i_a\neq i_b \text{ for } a\neq b}}\tau_{i_1}\dots \tau_{i_{n-1}}\otimes x_{i_11}\dots x_{i_{n-1}n-1}\\
&=\sum_{j=1}^n \sum_{(i_1,\dots,i_{n-1}) \in I_j} \tau_{i_1}\dots\tau_{i_{n-1}}\otimes x_{i_11}\dots x_{i_{n-1}n-1}.
\end{align*}
As in the proof of Lemma \ref{act}, we have $\tau_{i_1}\dots\tau_{i_{n-1}}=\tau_j$ for $(i_1, \dots, i_{n-1}) \in I_j$ and we get
\begin{align*}
\sum_{j=1}^n \tau_j \otimes x_{jn} = \sum_{j=1}^n \left(\tau_j \otimes \sum_{(i_1,\dots,i_{n-1}) \in I_j}x_{i_11}\dots x_{i_{n-1}n-1}\right).
\end{align*}
We deduce 
\begin{align*}
x_{jn} =  \sum_{(i_1,\dots,i_{n-1}) \in I_j}x_{i_11}\dots x_{i_{n-1}n-1},
\end{align*}
which is equivalent to Relation \eqref{7.5} by Lemma \ref{SO}. Thus, we also get a surjective map $C(SO_n^{-1}) \to C(G_{aut}^+(FQ_n))$ which is inverse to the map $C(G_{aut}^+(FQ_n)) \to C(SO_n^{-1})$.
\end{proof}

\,
\begin{rem}~
\begin{itemize}
\item[(i)] It was asked in \cite{survey} by Banica, Bichon and Collins to investigate the quantum automorphism group of the Clebsch graph. Since the $5$-folded cube graph is the Clebsch graph we get $G_{aut}^+(\Gamma_{Clebsch}) = SO_5^{-1}$.
\item[(ii)] The 3-folded cube graph is the full graph on four points, thus our theorem yields $S_4^+ = SO_3^{-1}$, as already shown in \cite{4points}. 
\end{itemize}
\end{rem}

\begin{rem}
We do not have a similar theorem for folded cube graphs $FQ_n$ with $n$ even, since the eigenspace $E_{-n+2}$ behaves different in the odd case. 
\end{rem}

\bibliographystyle{plain}
\bibliography{Clebschgraph}

\end{document}